\documentclass[reqno, 12pt, reqno]{amsart}

\usepackage{amsfonts, amsthm, amsmath, amssymb}
\usepackage{hyperref}
\hypersetup{colorlinks=false}

\usepackage[margin=1.5in]{geometry}

\usepackage{helvet}

\RequirePackage{mathrsfs} \let\mathcal\mathscr

\numberwithin{equation}{section}

\newtheorem{theorem}{Theorem}[section]
\newtheorem{lemma}[theorem]{Lemma}

\theoremstyle{definition}
\newtheorem*{ack}{Acknowledgements}

\renewcommand{\d}{\mathrm{d}}
\renewcommand{\phi}{\varphi}
\renewcommand{\rho}{\varrho}

\newcommand{\sumstar}{\sideset{}{^*}\sum}

\newcommand{\0}{\mathbf{0}}

\newcommand{\PP}{\mathbb{P}}
\renewcommand{\AA}{\mathbb{A}}

\newcommand{\FF}{\mathbb{F}}
\newcommand{\ZZ}{\mathbb{Z}}

\newcommand{\NN}{\mathbb{Z}_{>0}}

\newcommand{\RR}{\mathbb{R}}
\newcommand{\CC}{\mathbb{C}}

\newcommand{\TT}{\mathbb{T}}

\newcommand{\cO}{\mathcal{O}}

\renewcommand{\leq}{\leqslant}

\renewcommand{\geq}{\geqslant}

\renewcommand{\bar}{\overline}

\newcommand{\x}{\mathbf{x}}
\newcommand{\y}{\mathbf{y}}

\renewcommand{\v}{\mathbf{v}}
\renewcommand{\u}{\mathbf{u}}

\DeclareMathOperator{\Spec}{Spec}

\DeclareMathOperator{\tr}{Tr}

\DeclareMathOperator{\ord}{ord}

\DeclareMathOperator{\mor}{Mor}

\DeclareMathOperator{\ch}{char}

\newcommand{\vp}{\varpi}

\renewcommand{\hat}{\widehat}
\newcommand{\Ki}{K_\infty}

\newcommand{\hQ}{\widehat{Q}}

\newcommand{\matr}[4]{\left( \begin{matrix} #1 & #2 \\ #3 & #4 \end{matrix} \right) }

\begin{document}


\date{\today}

\title[Rational curves on smooth  hypersurfaces]
{Rational curves on smooth\\  hypersurfaces of low degree}

\author{T.D. Browning}
\address{School of Mathematics\\
University of Bristol\\ Bristol\\ BS8 1TW
\\ UK}
\email{t.d.browning@bristol.ac.uk}

\author{P. Vishe}
\address{Department of Mathematical Sciences\\ Durham University\\ Durham\\ DH1 3LE\\ UK}
\email{pankaj.vishe@durham.ac.uk}

\begin{abstract}
We study the family of rational curves on arbitrary smooth hypersurfaces of low degree using tools from analytic number theory.
\end{abstract}

\maketitle

\setcounter{tocdepth}{1}
\tableofcontents

\thispagestyle{empty}

\section{Introduction}

The geometry of a variety is intimately linked to the geometry of the space of rational curves on it. 
Given a  projective variety
  $X$
  defined over $\CC$, a natural object to study is 
the moduli space of rational curves on $X$.
There are many results in the literature establishing the irreducibility of such mapping spaces, but most   statements are only proved for generic $X$.
Following a strategy of Ellenberg and Venkatesh, 
we shall use tools from analytic number theory to prove such a result for all smooth hypersurfaces of sufficiently low degree.

Let  $X\subset \PP^{n}$ be a smooth Fano  hypersurface of degree $d$ defined over $\CC$, with $n\geq 3$. For each positive integer $e$, the 
 Kontsevich moduli space 
$\overline{\mathcal{M}}_{0,0}(X,e)$ is a compactification of the space 
${\mathcal{M}_{0,0}}(X,e)$  of morphisms of degree $e$ from $\PP^1$ to $X$, up to isomorphism.
According to Koll\'ar 
\cite[Thm.~II.1.2/3]{kollar-book},  any irreducible component of 
$\overline{\mathcal{M}}_{0,0}(X,e)$ has dimension at least
\begin{equation}\label{eq:bar-mu}
\bar\mu=(n+1  - d)e + n - 4.
\end{equation}
Work of Harris, Roth and Starr \cite{HRS} shows that  
$\overline{\mathcal{M}}_{0,0}(X,e)$
 is an irreducible, local complete intersection scheme of dimension $\bar\mu$, provided that $X$
 is general and $d<\frac{1}{2}(n+1)$.  
The restriction on $d$ has since been weakened to 
$d<\frac{2}{3}(n+1)$ by  Beheshti and Kumar \cite{BK} (assuming that $n\geq 23$),
and then to $d\leq n-2$ by Riedl and Yang \cite{RY}.

In the setting   $d=3$  of cubic hypersurfaces it is possible to obtain 
results for all smooth hypersurfaces in the family. 
Thus Coskun and Starr \cite{CS} have shown 
that  $\overline{\mathcal{M}}_{0,0}(X,e)$ is irreducible  and of  dimension $\bar\mu$
for any smooth cubic hypersurface  $X\subset \PP^n$ over $\CC$, provided that   $n> 4$.
(If $n=4$ then $\overline{\mathcal{M}}_{0,0}(X,e)$ has two irreducible components of the expected dimension 
$\bar\mu=2e$.)

At the expense of a much stronger condition on the degree, our main result establishes the irreducibility and dimension of the space 
${\mathcal{M}}_{0,0}(X,e)$,  for an arbitrary  smooth hypersurface $X\subset \PP^n$ over $\CC$.
Let
\begin{equation}\label{eq:n0}
n_0(d)=2^{d-1}(5d-4).
\end{equation}
We shall prove the following statement. 

\begin{theorem}\label{thm:irred}
Let $X\subset \PP^{n}$ be a smooth  hypersurface of degree $d\geq 3$ defined over $\CC$, with 
$n\geq n_0(d)$. 
Then  for each $e\geq 1$ the space
${\mathcal{M}}_{0,0}(X,e)$ is irreducible and of the expected dimension.
\end{theorem}

The example of Fermat hypersurfaces, discussed in \cite[\S 1]{CS},
shows  that the analogous result for 
$\overline{\mathcal{M}}_{0,0}(X,e)$ is false when $d>3$ and $e$ is large enough.
When $e=1$ we have $\overline{\mathcal{M}}_{0,0}(X,1)=\mathcal{M}_{0,0}(X,1)=F_1(X)$, where $F_1(X)$ is 
the  Fano scheme of lines on $X$. 
It has been conjectured, independently 
by Debarre and de Jong, that $\dim F_1(X)=2n-d-3$ for any  
smooth Fano hypersurface $X\subset \PP^n$ of degree $d$.
Beheshti \cite{Beh} has confirmed this  for  $d\leq 8$. 
Taking $e=1$ in Theorem \ref{thm:irred}, 
we conclude that $\dim F_1(X)=2n-d-3$ for any    $d\geq 3$, provided that
 $n\geq  n_0(d)$.  

Our proof of Theorem \ref{thm:irred}  ultimately relies on techniques from analytic number theory. 
The first step is ``spreading out'', in the sense of Grothendieck \cite[\S 10.4.11]{EGAIV} (cf.\ Serre \cite{serre}), 
which will  take us to the analogous problem 
for smooth hypersurfaces defined over the algebraic closure of a finite field. 
Passing to a finite field $\FF_q$ of sufficiently large cardinality, 
for a smooth degree $d$ hypersurface $X\subset\PP_{\FF_q}^{n}$ defined over $\FF_q$,
the cardinality of $\FF_q$-points
on $\mathcal{M}_{0,0}(X,e)$ can be related to the number of $\FF_q(t)$-points on $X$ of degree  $e$.
We shall access the latter quantity through a function field version of the Hardy--Littlewood circle method. A comparison with 
the Lang--Weil estimate \cite{LW} then allows us to make deductions about the irreducibility and dimension
of $\mathcal{M}_{0,0}(X,e)$.

The idea of  using the circle method to study the moduli space of rational curves on varieties is due to  Ellenberg and Venkatesh.   
The traditional setting for the circle method is a fixed finite field $\FF_q$, with the goal being to understand the 
$\FF_q(t)$-points on $X$ of degree  $e$, as $e\rightarrow \infty$. 
This is the point of view taken in  work of Lee \cite{lee,lee'} 
on a $\FF_q(t)$-version of Birch's work on systems of forms in many variables.
In  contrast to this, we will be required to  handle any fixed $e\geq 1$, as $q\rightarrow \infty$.
Pugin developed  an ``algebraic circle method'' 
in his 2011 Ph.D. thesis \cite{pugin} 
to study the 
spaces 
$\mathcal{M}_{0,0}(X,e)$, when
$X\subset \PP_{\FF_q}^{n}$ is the diagonal cubic hypersurface
$$
a_0x_0^3+\dots+a_nx_n^3=0,
\quad (\text{for $a_0,\dots,a_n\in \FF_q^*$}).
$$
Assuming that   $n\geq 12$ and $\ch(\FF_q)> 3$, he succeeds in showing that 
the   space
$\mathcal{M}_{0,0}(X,e)$  is   irreducible  and of the expected dimension.
Our work, on the other hand,  applies  to arbitrary smooth hypersurfaces of sufficiently low degree, which are defined over the complex numbers.
Finally, our investigation bears comparison with work of Bourqui \cite{bourqui1, bourqui2}, who has also  investigated the moduli space of curves on  varieties using  counting arguments. In place of the circle method, however, Bourqui draws on the theory of  universal torsors.

\begin{ack}
The authors are grateful to 
Hamid Ahmadinezhad,  
Lior Bary Soroker, 
Roya Beheshti, 
Emmanuel Peyre 
and the anonymous referee 
for useful  conversations. 
While working on this paper the first  author was
supported by ERC grant \texttt{306457}.
\end{ack}

\section{Spreading out}

Let $X\subset \PP^n$ be a smooth hypersurface of degree $d$, defined by a homogeneous polynomial 
$$
F(x_0,\dots,x_n)=\sum_{\substack{\mathbf{i}\in \ZZ_{\geq 0}^{n+1}\\ i_0+\dots+i_n=d}} c_{\mathbf{i}} 
x_0^{i_0}\dots x_n^{i_n},
$$
with coefficients $c_{\mathbf{i}}\in \CC$.
Rather than working with 
$\mathcal{M}_{0,0}(X,e)$, 
it will suffice to study the naive space 
$\mor_e(\PP^1, X)$ of actual maps $\PP^1\to X$ of degree $e$. 
The expected dimension of 
$\mor_e(\PP^1, X)$ is $\mu=\bar\mu+3$, 
where $\bar\mu$ is given by \eqref{eq:bar-mu}, 
since $\PP^1$ has automorphism group of dimension $3$.
We proceed to recall the construction of 
$\mor_e(\PP^1, X)$.

Let  $G_e$ be the set of all homogeneous polynomials in $u,v$ of degree $e\geq 1$, with coefficients in $\CC$.
A {\em rational curve} of degree $e$ on $X$ is a non-constant morphism $f: \PP_\CC^1\rightarrow X$ of degree $e$. It is 
given by 
$$
f=(f_0(u,v),\dots,f_n(u,v)),
$$
with $f_0,\dots,f_n\in G_e$, with no non-constant common factor in $\CC[u,v]$, such that $F(f_0(u,v),\dots,f_n(u,v))$ vanishes  identically.
We may regard $f$ as a point in $\PP_\CC^{(n+1)(e+1)-1}$ and 
the morphisms of degree $e$ on $X$ are parameterised by 
$\mor_e(\PP_\CC^1,X)$, which is an open subvariety of 
$\PP_\CC^{(n+1)(e+1)-1}$ cut out by a system of $de+1$ equations of degree $d$.
In this way we obtain the expected dimension
$$
(n+1)(e+1)-1-(de+1)=(n+1-d)e+n-1=\mu,
$$
of  
$\mor_e(\PP_\CC^1,X)$.
It follows from   Koll\'ar \cite[Thm.~II.1.2]{kollar-book} that  all irreducible components 
of $\mor_e(\PP_\CC^1, X)$ 
have  dimension at least $\mu$.
In order to  establish Theorem~\ref{thm:irred} it will therefore suffice to show that 
$\mor_e(\PP_\CC^1,X)$  is irreducible, with 
$\dim\mor_e(\PP_\CC^1,X)\leq \mu,$
provided that $n\geq   n_0(d)$.

The complement to 
$\mor_e(\PP_\CC^1,X)$ in its closure is the set of $(f_0,\dots,f_n)$ with a common zero.
We can obtain explicit equations for 
$\mor_e(\PP_\CC^1,X)$  by noting that 
$f_0,\dots,f_n$ have a common zero if and only if the resultant 
$\mathrm{Res}(\sum_i \lambda_i f_i, \sum_j \mu_j f_j)$ is identically zero as a polynomial in $\lambda_i,\mu_j$. 
It is clear that both $X$ and $\mor_e(\PP_\CC^1,X)$ are defined by equations 
with coefficients belonging to the 
finitely generated $\ZZ$-algebra
$\Lambda=\ZZ[c_{\mathbf{i}}]$,
obtained by adjoining the coefficients of $F$ to $\ZZ$. 
In this way we may view $X$ and 
 $\mor_e(\PP_\CC^1,X)$ as schemes over $\Lambda$, with structure morphisms
$ X\to \Spec \Lambda$ and 
$$
  \mor_e(\PP_\CC^1,X)\to \Spec \Lambda.
$$
By Chevalley's upper semicontinuity theorem (see \cite[Thm.~13.1.3]{EGAIV}), 
there exists a non-empty open set 
$U$ of $\Spec \Lambda$ such that 
$$
\dim \mor_e(\PP_\CC^1,X)\leq \dim
\mor_e(\PP_\CC^1,X)_{\mathfrak{m}}
$$
for any closed point  $\mathfrak{m}\in U$. Here  $\mor_e(\PP_\CC^1,X)_\mathfrak{m}$ denotes the fibre above  $\mathfrak{m}$, which is obtained  via the base change  $\Spec \Lambda/\mathfrak{m}\to \Spec \Lambda$.
Likewise, since integrality
is an open condition, 
the space 
$\mor_e(\PP_{\CC}^1,X)$ will be irreducible if  
$\mor_e(\PP_\CC^1,X)_{\mathfrak{m}}$ is.

Choose a maximal ideal $\mathfrak{m}$ in $U$.
The quotient  $\Lambda/\mathfrak{m}$ is a finite field by arithmetic weak Nullstellensatz. 
By enlarging $\Lambda$,  we may assume that it contains $1/d!$. In particular, it follows that 
$\ch(\Lambda/\mathfrak{m})=p$, say, with $p>d$, since any prime less than or equal to $d$ is invertible in $\Lambda$.
The quasi-projective  varieties  $X_\mathfrak{m}$ and 
$\mor_e(\PP_\CC^1,X)_\mathfrak{m}$ are defined over $\bar \FF_p$, being 
given explicitly  by 
 reducing modulo $\mathfrak{m}$ the coefficients of  the original system of defining equations.
By further enlarging $\Lambda$, if necessary, we may assume that $X_{\mathfrak{m}}$ is smooth.
There exists a finite field $\FF_{q_0}$ such that $X_\mathfrak{m}$ and 
$\mor_e(\PP_\CC^1,X_{\CC})_{\mathfrak{m}}$ are both 
defined over $\FF_{q_0}$. In view of the Lang--Weil estimate,  Theorem \ref{thm:irred} is a direct consequence of the  following result, together with  the fact that $\mor_e(\PP_\CC^1,X_{\CC})_{\mathfrak{m}}$ 
is non empty in the cases under consideration. 

\begin{theorem}\label{thm:LW}
Let  
$n\geq   n_0(d)$
and let 
 $X\subset \PP_{\FF_q}^{n}$ be a smooth hypersurface of degree $d\geq 3$ defined over a finite field $\FF_q$, with $\ch(\FF_q)>d$.
Then  for  each $e\geq 1$ we have 
$$
\lim_{\ell\to \infty}q^{-\ell \mu}\#\mor_e(\PP_{\FF_q}^1, X)(\FF_{q^\ell})\leq 1.
$$
\end{theorem}

\section{The Hardy--Littlewood circle method}

We now initiate the proof of Theorem \ref{thm:LW}.
We henceforth redefine $q^\ell$ to be $q$ and we replace $n$ by $n-1$ in the statement of the theorem. In particular the expected dimension is now
$\mu=(n-d)e+n-2$.
Our proof of Theorem~\ref{thm:LW} is based on a version of the Hardy--Littlewood circle method for the function field $K=\FF_q(t)$,
always under the assumption that 
$\ch(\FF_q)>d$. 
The main input for this  comes from work of Lee \cite{lee, lee'}, combined with our 
own recent contribution to the subject,  in the setting of cubic forms  \cite{BV}.

We begin by laying down some basic notation and terminology.
To begin with, for any real number $R$ we set $\hat R=q^R$.
Let 
$\mathcal{O}=\FF_q[t]$ be the ring of integers of $K$ and 
let $\Omega$ be the set of  places of $K$. These correspond to either monic irreducible polynomials $\varpi$ in $\mathcal{O}$, which we call the {\em finite primes},  or the {\em prime at infinity} $t^{-1}$ which we usually denote  by $\infty$. 
The associated absolute value  $|\cdot|_v$ is either $|\cdot|_\vp$ for some prime $\vp\in \mathcal{O}$ or $|\cdot|$, according to whether $v$ is a finite or infinite place, respectively. 
These  are given by 
$$
|a/b|_\vp=\left(\frac{1}{q^{\deg \vp}}\right)^{\ord_\vp(a/b)} \quad \text{ and }\quad
|a/b|=
q^{\deg a-\deg b},
$$
for any $a/b\in K^*$.
We extend these definitions to  $K$ by taking $|0|_\vp=|0|=0.$

For $v\in \Omega$ we let $K_v$ denote the completion of $K$ at $v$ with respect to $|\cdot|_v$. 
We may identify $K_\infty$ with the set 
$$
\FF_q((1/t))=\left\{\sum_{i\leq N}a_it^i: \mbox{for $a_i\in \FF_q$ and some $N\in\ZZ$} \right\}.
$$
We can extend the absolute value at the infinite place to $K_\infty$ to get  a non-archimedean
absolute value 
$|\cdot|:K_\infty\rightarrow \RR_{\geq 0}$ 
given by $|\alpha|=q^{\ord \alpha}$, where $\ord \alpha$ is the largest $i\in
\ZZ$ such that $a_i\neq 0$ in the 
representation $\alpha=\sum_{i\leq N}a_it^i$.
In this context we adopt the convention $\ord 0=-\infty$ and $|0|=0$.
We extend this to vectors by setting 
$
|\x|=\max_{1\leq i\leq n}|x_i|, 
$
for any $\x\in K_\infty^n$. 

Next, we  put 
$$
\TT=\{\alpha\in K_\infty: |\alpha|<1\}=\left\{\sum_{i\leq -1}a_it^i: \mbox{for $a_i\in \FF_q$}
\right\}.
$$
Since $\TT$ is a locally compact
additive subgroup of $K_\infty$ it possesses a unique Haar measure $\d
\alpha$, which is normalised so that 
$
\int_\TT \d\alpha=1.
$
We can extend $\d\alpha$ to a (unique) translation-invariant measure on $\Ki$, in
such a way that 
$$
\int_{\{\alpha\in\Ki: 
|\alpha|<\hat N
\}} \d \alpha=\hat N,
$$ 
for any $N\in \NN$.
These measures also extend to $\TT^n$ and $\Ki^n$, for any $n\in \NN$.
There is a non-trivial additive character $e_q:\FF_q\rightarrow \CC^*$ defined
for each $a\in \FF_q$ by taking 
$e_q(a)=\exp(2\pi i \tr_{\FF_q/\FF_p}(a)/p)$.
This character yields a non-trivial (unitary) additive character $\psi:
K_\infty\rightarrow \CC^*$ by defining $\psi(\alpha)=e_q(a_{-1})$ for any 
$\alpha=\sum_{i\leq N}a_i t^i$ in $\Ki$.

Let $F\in \FF_q[\x]$ be a non-singular form of degree $d\geq 3$,
with 
$\x=(x_1,\dots,x_n)$. We may express this polynomial as
$$
F(\x)
=\sum_{i_1,\dots,i_d=1}^n c_{i_1,\dots,i_d} 
x_{i_1}\dots x_{i_d},
$$
with coefficients $c_{i_1,\dots,i_d} \in \FF_q$.
 In particular $F$ and the discriminant 
 $\Delta_F$ are non-zero, or equivalently,   
$\max_{\mathbf{i}}|c_{\mathbf{i}}|=1$ 
and 
$|\Delta_F|=1.$
We will make frequent use of these facts in what follows. 
Associated to $F$ are the multilinear forms
\begin{equation}\label{eq:multi}
\Psi_i(\x^{(1)},\dots,\x^{({d-1})})
=\sum_{i_1,\dots,i_{d-1}=1}^n c_{i_1,\dots,i_{d-1},i} 
x_{i_1}^{(1)}\dots x_{i_{d-1}}^{(d-1)},
\end{equation}
for $1\leq i\leq n$.

To establish Theorem~\ref{thm:LW}  we work with the  naive space
$$
M_e=\left\{ \x=(x_1,\dots,x_n)\in G_e(\FF_q)^{n}\setminus\{\0\} : F(\x)=0\right\},
$$
where $G_e(\FF_q)$ is the set of binary forms of degree $e$ with coefficients in $\FF_q$. Thus $M_e$
 corresponds to the  $\FF_q$-points on the affine cone  of 
 $\mor_e(\PP_{\FF_q}^1,X)$, where we drop the condition  
 that $x_1,\dots,x_n$ share no common  factor. 
Let us set 
\begin{equation}\label{eq:hat-mu}
\begin{split}
\hat \mu=\mu+1
&=(n-d)e+n-1=(e+1)n-de-1.
\end{split}
\end{equation}
It will clearly suffice to show that 
\begin{equation}\label{eq:goal}
\lim_{q\to \infty}q^{-\hat \mu}\#M_e\leq 1,
\end{equation}
for $n> n_0(d)$, where $n_0(d)$ is given  by  \eqref{eq:n0}.
We proceed by relating  $\#M_e$ to the counting function that lies at the heart of our earlier investigation \cite{BV}.

Let $w:K_\infty^n\to \{0,1\}$ be given by
$w(\x)=
\prod_{1\leq i\leq n} w_\infty(x_i),
$
where 
$$
w_\infty(x)=
\begin{cases}
1, &\text{if $|x|<1$,}\\
0, & \text{otherwise}.
\end{cases}
$$
Putting $P=t^{e+1}$, we then have
$
\#M_e\leq N(P),
$
where
$$
N(P)=\sum_{\substack{\x\in \cO^n\\ F(\x)=0}} w(\x/P).
$$
It follows from \cite[Eq.~(4.1)]{BV} that for any $Q\geq 1$ we have 
\begin{equation}\label{eq:goat}
N(P)=
\sum_{\substack{
r\in \cO\\
|r|\leq \hat Q\\
\text{$r$ monic}
} }
\sumstar_{\substack{
|a|<|r|} }
\int_{|\theta|<|r|^{-1}\hQ^{-1}} S\left(\frac{a}{r}+\theta\right) \d \theta,
\end{equation}
where 
$
\sum^*$ means that the sum is taken over residue classes 
$|a|<|r|$ for which  $\gcd(a,r)=1$, and where
\begin{equation}\label{eq:SUM}
S(\alpha)=\sum_{\substack{\x\in \cO^n}} \psi(\alpha F(\x))w(\x/P),
\end{equation}
for any $\alpha\in \TT$.
We will work with the choice $Q=d(e+1)/2$, so that 
$
\hat Q=|P|^{d/2}.
$

The {\em major arcs} for our problem are given by $r=1$ and $|\theta |<|P|^{-d} q^{d-1}$. 
We let the {\em minor arcs} be everything else: i.e.\ those $\alpha=a/r+\theta$ 
appearing in \eqref{eq:goat} for which 
either $|r|>q$, or else $r=1$ and 
$|\theta |\geq |P|^{-d} q^{d-1}$. 
The contribution $N_{\text{major}}(P)$ from the major arcs is easy to deal with.
Indeed, for $|\theta|<|P|^{-d} q^{d-1}$ and $|\x|<|P|$ we have 
$|\theta F(\x)|<|P|^{-d}q^{d-1}q^{de}=q^{-1}$, whence $\psi(\theta F(\x))=1$. Thus $S(\alpha)=|P|^n$, for $\alpha=\theta$ belonging to the major arcs, whence
$$
N_{\text{major}}(P)=|P|^n \int_{|\theta|<|P|^{-d}q^{d-1}}\d \theta = |P|^{n-d}q^{d-1}=q^{\hat \mu}.
$$
In order to prove \eqref{eq:goal}, it therefore remains to show that 
\begin{equation}\label{eq:cup}
\lim_{q\to\infty} q^{-\hat \mu}N_{\text{minor}}(P)=0,
\end{equation}
for $n>n_0(d)$, where 
$N_{\text{minor}}(P)$ is the overall contribution to \eqref{eq:goat} from the minor arcs.
This will complete  the proof of Theorem \ref{thm:LW}.

\section{Geometry of numbers in function fields}

The purpose of this section is to record a technical result about lattice point counting over $K_\infty$.
A {\em lattice} in $K_\infty^N$ is a set of points of the form $\x=\Lambda \u$, where $\Lambda$ is a $N\times N$ matrix over $K_\infty$ and $\u$ runs over elements of $\cO^N$. By an abuse of notation we will also denote the set of such points by $\Lambda$.
Given a lattice $M$, the {\em adjoint lattice} $\Lambda$ is defined to satisfy $\Lambda^T M=I_N$, where $I_N$ is the $N\times N$ identity matrix. 

Let $\gamma=(\gamma_{ij})$ 
be a symmetric $n\times n$ matrix with entries in $K_\infty$. Given any positive integer $m$, we  define the special lattice
$$
M_m=\matr{t^{-m}I_n}{0}{t^m\gamma}{t^mI_n},
$$
with corresponding adjoint lattice
$$
\Lambda_m =\matr{t^{m}I_n}{-t^m\gamma}{0}{t^{-m}I_n}.
$$
Let $\hat R_1,...,\hat R_{2n}$ denote the successive minima of the
lattice corresponding to $M_m$.
For any vector $\x\in K_\infty^{2n}$ let $\x_1=(x_1,\dots,x_{n})$ and $\x_2=(x_{n+1},\dots,x_{2n})$.
We claim that $M_m$ and $\Lambda_m$ can be identified with one 
another. Now $M_m$ is the set of points $\x=M_m\u$ where $\u=(\u_1,\u_2)$ runs over elements of $\cO^{2n}$. Likewise,
 $\Lambda_m$ is the set of points $\y=\Lambda_m\v$ where $\v=(\v_1,\v_2)$ runs over elements of $\cO^{2n}$. 
We can therefore identify $M_m$ with $\Lambda_m$ through the process of changing the sign of $\v_2$, then the sign of $\y_2$, then switching $\v_1$ with $\v_2$, and finally interchanging $\y_1$ and $\y_2$.
It now follows from \cite[Lemma 3.3.6]{lee} (cf.~\cite[Lemma B.6]{lee'}) 
that 
\begin{equation}\label{eq:cime}
R_\nu+R_{2n-\nu+1}=0,
\end{equation}
for $1\leq \nu\leq n$.
An important step in the proof of \cite[Lemma~3.3.6]{lee} 
(cf.~\cite[Lemma B.6]{lee'}) 
 is a non-archimedean version of Gram--Schmidt orthogonalisation, which is used without reference in the  proof of \cite[Lemma 3.3.3]{lee}
 (cf.~\cite[Lemma B.3]{lee'}). 
This deficit is remedied by appealing to recent work of Usher and Zhang \cite[Theorem 2.16]{UZ}.

For any $Z\in \RR$ and any lattice $\Gamma$ we define the counting function
\begin{align*}
\Gamma(Z)= \#\{\x\in\Gamma:| \x|< \hat Z\}.
\end{align*}
Note that $\Gamma(Z)=\Gamma(\lceil Z \rceil)$ for any $Z\in \RR$.
We proceed to establish the following inequality.
\begin{lemma}
\label{lem:M ratio}
 Let $m, Z_1,Z_2\in \ZZ$ such that $Z_1\leq Z_2\leq 0$. Then we have 
 \begin{align*}
  \frac{M_m(Z_1)}{M_m(Z_2)}\geq \left(\frac{\hat Z_1}{\hat Z_2}\right)^n.
 \end{align*}
\end{lemma}
\begin{proof}
Let $1\leq\mu,\nu\leq 2n$ be such that $R_{\mu}< Z_1\leq R_{\mu+1}$ and $R_{\nu}< Z_2\leq R_{\nu+1} $. 
Since $R_j$ is a non-decreasing sequence which satisfies $R_j+R_{2n-j+1}=0$, 
by \eqref{eq:cime}, 
we must  have $0\leq R_{n+1}$, whence  in fact $\mu\leq \nu\leq n$.   
It  follows from 
\cite[Lemma~3.3.5]{lee}  
(cf.~\cite[Lemma B.5]{lee'}) 
that
  \begin{align*}\frac{M_m(Z_1)}{M_m(Z_2)}=
    \begin{cases}
1 & \textrm{ if }Z_1,Z_2<R_1, \\
\left(\prod_{j=1}^\nu\hat R_j/\hat{Z_1}\right)(\hat{Z_1}/\hat{Z_2})^{\nu} & \textrm{ if }Z_1<R_1\leq Z_2, \\
\left(\prod_{j=\mu+1}^\nu\hat R_j/\hat{Z_1}\right)(\hat{Z_1}/\hat{Z_2})^{\nu}  & \textrm{ if }R_1\leq Z_1\leq Z_2 ,
\end{cases}
  \end{align*}
The statement of the lemma is now obvious.
\end{proof}

As above, let $\gamma=(\gamma_{ij})$ be a symmetric $n\times n$ matrix with entries in $K_\infty$. 
For $1\leq i\leq n$  we introduce the  linear forms 
$$
 L_i(u_1,\dots,u_n)=\sum_{j=1}^n \gamma_{ij}u_j.
$$
Next, for given real numbers $a, Z$, we let $N(a,Z)$ denote the number of vectors $(u_1,\dots,u_{2n})\in
\cO^{2n}$ such that 
\begin{align*}
  |u_j|<\hat{a}\hat{Z}\quad \textrm{ and }\quad |L_j(u_1,\dots,u_n)+u_{j+n}|<\frac{\hat{Z}}{\hat{a}}\quad \text{ for $1\leq j\leq n$}.
\end{align*}
If we put $m=\lfloor a\rfloor$, then it is clear that 
$$
M_m(Z-\{a\})\leq 
N(a,Z)\leq M_m(Z+\{a\}),
$$
where $\{a\}$ denotes the fractional part of $a$.
The following result is a direct consequence of Lemma 
\ref{lem:M ratio}.

\begin{lemma}\label{lem:cape}
 Let $a, Z_1,Z_2\in \RR$ with  $Z_1\leq Z_2\leq 0$. Then we have 
 \begin{align*}
 \frac{N(a,Z_1)}{N(a,Z_2)}\geq   \hat K^n,
 \end{align*}
where 
$
K=\lceil Z_1-\{a\}\rceil-\lceil Z_2+\{a\}\rceil.
$
\end{lemma}

\section{Weyl differencing}

In everything that follows we shall 
assume that $\ch(\FF_q)>d$ and we will allow all our implied constants to depend at most on $d$ and $n$. 
This section is concerned with a careful analysis of  the exponential sum  
 \eqref{eq:SUM}, using
the function field version of  Weyl differencing that was worked out by Lee \cite{lee,lee'}. 
Our task is  to make the dependence on $q$ completely explicit and it turns out that gaining satisfactory control requires considerable care.
Since we are  concerned with hypersurfaces one needs to take $R=1$  in  \cite{lee, lee'}.  

For any $\beta=\sum_{i\leq N}b_i t^i \in K_\infty$, 
we let $\|\beta\|=|\sum_{i\leq -1}b_it^i|$.
Recalling the definition \eqref{eq:multi} of the multilinear forms associated to $F$, we
let 
\begin{equation}\label{eq:defineN}
N(\alpha)=
\#\left\{ \underline{\u}\in \cO^{(d-1)n}: 
\begin{array}{l}
|\u_1|,\dots,|\u_{d-1}|<|P|\\
\|\alpha \Psi_i(\underline{\u}) \|<|P|^{-1} ~(\forall   i\leq n)
\end{array}{}
\right\},
\end{equation}
where $\underline{\u}=(\u_1,\dots,\u_{d-1})$.
We begin with an application of \cite[Cor.~4.3.2]{lee}
(cf.~\cite[Cor.~3.3]{lee'}), 
which leads to the inequality
\begin{equation}\label{eq:monday}
|S(\alpha)|^{2^{d-1}}\leq |P|^{(2^{d-1}-d+1)n}N(\alpha),
\end{equation}
for any $\alpha\in \TT$.

The next stage in the analysis of $S(\alpha)$ is a multiple application of the function field analogue of Davenport's ``shrinking lemma'',  as proved in \cite[Lemma~4.3.3]{lee}
(cf.~\cite[Lemma 3.4]{lee'}), 
ultimately leading to \cite[Lemma 4.3.4]{lee}
(cf.~\cite[Lemma 3.5]{lee'}). 
Unfortunately the implied constant in these estimates is allowed to depend on $q$ and so we must work   harder to control it.
Let
$$
N_\eta(\alpha)=
\#\left\{ \underline{\u}\in \cO^{(d-1)n}: 
\begin{array}{l}
|\u_1|,\dots,|\u_{d-1}|<|P|^\eta\\
\|\alpha \Psi_i(\underline{\u}) \|<|P|^{-d+(d-1)\eta} ~(\forall   i\leq n)
\end{array}{}
\right\},
$$
for any parameter $\eta\in [0,1].$  Recalling that $P=t^{e+1}$,  we shall prove the following uniform version  of \cite[Lemma 4.3.4]{lee}
(cf.~\cite[Lemma 3.5]{lee'}).

\begin{lemma}\label{lem:3.5}
Let  $\alpha\in \TT$ and suppose that $\eta\in [0,1)$ is chosen so that 
\begin{equation}\label{eq:hypothesis}
\frac{(e+1)(\eta+1)}{2}\in \ZZ.
\end{equation}
Then we have 
$
N(\alpha)\leq 
|P|^{(n-\eta n)(d-1)}N_{\eta}(\alpha).
$
In particular, we have 
$$
|S(\alpha)|^{2^{d-1}}\leq \frac{|P|^{2^{d-1}n}}{|P|^{\eta (d-1)n}}N_\eta(\alpha).
$$
\end{lemma}

\begin{proof}
In view of \eqref{eq:defineN} and  \eqref{eq:monday}, 
the final part follows from the first part. 
For each $v\in \{0,\dots,d-1\}$, 
define $N^{(v)}(\alpha)$ to be the number of $\underline{\u}\in \cO^{(d-1)n}$
such that 
\begin{equation}\label{eq:lapel}
|\u_1|,\dots,|\u_{v}|<|P|^\eta, \quad 
|\u_{v+1}|,\dots,|\u_{d-1}|<|P| 
\end{equation}
and 
$
\|\alpha \Psi_i(\underline{\u}) \|<|P|^{-v-1+v\eta},
$
for $1\leq i\leq n$.
Thus we have $N^{(0)}(\alpha)=N(\alpha)$ and $N^{(d-1)}(\alpha)=N_\eta(\alpha)$. It will suffice to show that 
\begin{equation}\label{eq:**}
N^{(v)}(\alpha)\geq 
|P|^{-n+\eta n}
N^{(v-1)}(\alpha),
\end{equation}
for each $v\in \{1,\dots,d-1\}$.

Fix a choice of $v$, together with $\u_1,\dots,\u_{v-1},\u_{v+1},\dots,\u_{d-1}\in \cO^n$ such that \eqref{eq:lapel} holds.
For each $1\leq i\leq n$ we consider the linear form 
\begin{align*}
L_i(\u)
&=\alpha \Psi_i(\u_1,\dots,\u_{v-1},\u,\u_{v+1},\dots,\u_{d-1})
=\sum_{j=1}^n \gamma_{ij}u_j,
\end{align*}
say, for a suitable 
symmetric $n\times n$ matrix $\gamma=(\gamma_{ij})$,  with entries in $K_\infty$. 
Given real numbers $a$ and $Z$, define  $N(a,Z)$ to be the number of  vectors $(u_1,\dots,u_{2n})$ in $\cO^{2n}$ satisfying 
\begin{align*}
  |u_j|<\hat{Z+a} \quad \text{ and }\quad |L_j(u_1,\dots,u_n)-u_{j+n}|<\hat{Z-a},  \quad \text{ for $1\leq j\leq n$}.
\end{align*}

We are interested in estimating   the number of $\u\in \cO^n$ such that $|\u|<|P|^\eta$ and 
$\|L_i(\u)\|<|P|^{-v-1+v\eta}$, for $1\leq i\leq n$, in terms of the  number of 
 $\u\in \cO^n$ such that $|\u|<|P|$ and 
$\|L_i(\u)\|<|P|^{-v+(v-1)\eta}$, for $1\leq i\leq n$. 
That is, we wish to compare 
$N(a,Z_1)$ with $N(a,Z_2)$, where 
$$\hat a=|P|^{(v+1-(v-1)\eta)/2}, \quad \hat{Z_1}=|P|^{(v+1)(\eta-1)/2},
\quad\hat{Z_2}=|P|^{(v-1)(\eta-1)/2}. 
$$
Note that $\hat a \hat Z_1=|P|^\eta$ and $\hat a \hat Z_2=|P|$.
Moreover, our hypothesis \eqref{eq:hypothesis} implies that 
$$
a=\frac{(e+1)(v+1)}{2}-\frac{(v-1)(e+1)\eta}{2}
=v(e+1)-
\frac{(v-1)(e+1)(\eta+1)}{2}
\in \ZZ.
$$
Similarly, \eqref{eq:hypothesis} implies that $Z_1,Z_2\in \ZZ$.
It now follows from Lemma \ref{lem:cape} that 
$$
\frac{N(a,Z_1)}{N(a,Z_2)}\geq 
\left(\widehat{ Z_1- Z_2 }\right)^n=|P|^{-n+\eta n},
$$
which thereby completes the proof of \eqref{eq:**}.
\end{proof}

Lemma \ref{lem:3.5} doesn't allow us  to handle the case $e=1$ of lines.  To circumvent this difficulty we shall invoke 
a simpler version of the shrinking lemma, as follows.

\begin{lemma}\label{lem:3.5 e=1}
Let  $\alpha\in \TT$ and let   $ v\in \{1,\dots,d\}$. Then we have 
$$
|S(\alpha)|^{2^{d-1}}\leq |P|^{(2^{d-1}-d+1)n}
q^{e(v-1)n}M^{(v)}(\alpha),
$$
where
$M^{(v)}(\alpha)$ is the number of 
$\underline{\u}\in \cO^{(d-1)n}$ such that 
$$
|\u_1|,\dots,|\u_{v-1}|<q, \quad 
|\u_{v}|,\dots,|\u_{d-1}|<|P|
$$
and 
$\|\alpha \Psi_i(\underline{\u}) \|<|P|^{-1}$ for $1\leq i\leq n$.
\end{lemma}
\begin{proof}
Noting that $N(\alpha)=M^{(1)}(\alpha)$, it follows from \eqref{eq:monday} that it will be enough to prove that $M^{(v-1)}(\alpha)\leq q^{en} M^{(v)}(\alpha)$ for  $2\leq v\leq d$. The proof follows that of Lemma \ref{lem:3.5} and so 
we shall be brief. Let $\u_1,\dots,\u_{v-1},\u_{v+1},\dots,\u_{d-1}\in \cO^n$ be vectors satisfying 
\begin{equation*}
|\u_1|,\dots,|\u_{v-1}|<q, \quad 
|\u_{v+1}|,\dots,|\u_{d-1}|<|P|. 
\end{equation*}
Let $\gamma$ and $N(a,Z)$ be as in the proof of Lemma \ref{lem:3.5}, corresponding to this choice. 
Lemma \ref{lem:cape} clearly implies that 
$$\frac{N(e+1,-e)}{N(e+1,0)}\geq q^{-en}. 
$$
However, $N(e+1,-e)$ denotes the number of 
$\u\in \mathcal{O}^n$ such that 
$|\u|<q$ and  $\|L_i(\u)\|<q^{-2e-1}$, for $1\leq i\leq n$. The lemma follows on
noting that $q^{-2e-1}<q^{-e-1}=|P|^{-1}$.
\end{proof}

The next step is an application of the function field analogue of Heath-Brown's Diophantine approximation lemma, as worked out in \cite[Lemma 4.3.5]{lee}
(cf.~\cite[Lemma 3.6]{lee'}).  
Let $\alpha=a/r+\theta$, where $a/r\in K$ and $\theta\in \TT.$ Note that the maximum absolute value of the coefficients of each multilinear form $\Psi_j$ is $1$. 
We shall apply 
\cite[Lemma 4.3.5]{lee} 
(cf.~\cite[Lemma 3.6]{lee'}) 
with $\hat M=|P|^{(d-1)\eta}$ and $\hat Y=|P|^{d-(d-1)\eta}$.
We want a maximal choice of $\eta\geq 0$ such that 
$$
|P|^{(d-1)\eta}<\min\left\{ |P|^{d-1}, \frac{1}{|r\theta|}, \frac{|P|^d}{|r|}\right\}
$$
and
$$
|P|^{(d-1)\eta}\leq |r|\max\left\{ 1, |P^d\theta|\right\}.
$$
This leads to the constraint $(e+1)\eta\leq \Gamma$, where
\begin{equation}\label{eq:Gamma}
\Gamma=\frac{1}{d-1}\ord\left(\min\left\{
\frac{|P|^{d-1}}{q},~ \frac{1}{q |r\theta|}, ~
\frac{|P|^d}{q|r|}, ~|r|\max\left\{ 1, |P^d \theta|\right\} 
\right\}\right),
\end{equation}
in which we abuse notation and denote by $\ord$ the integer exponent of $q$ that appears.
For $i\in \{0,1\}$, let $[\Gamma]_i$ denote the largest non-negative integer not exceeding $\Gamma$, which is congruent to $i$ modulo $2$.
We then choose  $\eta$ via
\begin{equation}\label{eq:eta}
(e+1)\eta= \begin{cases}
[\Gamma]_0 & \text{ if $2\nmid e$,}\\
[\Gamma]_1 & \text{ if $2\mid e$.}
\end{cases}
\end{equation}
One notes that $(e+1)\eta\leq \Gamma$ and \eqref{eq:hypothesis} is satisfied.

It now
follows from 
\cite[Lemma 4.3.5]{lee} 
(cf.~\cite[Lemma 3.6]{lee'}) 
that 
$
N_\eta(\alpha)\leq U_\eta,
$
where $U_\eta$ denotes the number of 
$\underline{\u}\in \cO^{(d-1)n}$ such that 
$|\u_1|,\dots,|\u_{d-1}|<|P|^\eta$ and 
$$
\Psi_i(\underline{\u})=0, \quad \text{ for $1\leq i\leq n$}.
$$
A standard calculation, which we recall here for completeness, now 
shows that the latter system of equations defines an affine variety 
$V\subset \AA^{(d-1)n}$ of dimension 
at most $(d-2)n$.
To see this, we note that the intersection of $V$ with the diagonal $\Delta=\{\underline\u\in \AA^{(d-1)n}: 
\u_1=\dots=\u_{d-1}\}$ is contained in the singular locus of $F$ and so has affine dimension $0$. The claim follows on noting that 
$0=\dim(V\cap \Delta)\geq \dim V+\dim \Delta -(d-1)n=\dim V-(d-2)n$ .

We now apply \cite[Lemma~2.8]{BV}. Since $|P|^\eta=q^{(e+1)\eta}$, with $(e+1)\eta\in \ZZ$, this  directly yields the existence of a positive constant $c_{d,n}$, independent of $q$,  such that  
$
U_\eta\leq c_{d,n} 
|P|^{\eta (d-2)n}$.
Inserting this into Lemma~\ref{lem:3.5}, we therefore arrive at the following conclusion.

\begin{lemma}
\label{lem:pointwise}
Let $L=2^{-d+1}n$, let  $a/r\in K$  and let $\theta\in \TT$. 
Let  $\eta$ be given by \eqref{eq:eta}.
Then  there exists a constant $c_{d,n}>0$, independent of $q$, such that 
$$
|S(a/r+\theta)|\leq c_{d,n} |P|^{n-L\eta}.
$$
\end{lemma}

It turns out that this  estimate is  inefficient when $|r|$ is small. 
Let
\begin{equation}\label{eq:kappa}\kappa=\begin{cases}
1 &\text{ if $2\nmid e$,}\\
0 &\text{ if $2\mid e$.}
\end{cases}
\end{equation}
It will also be advantageous to  consider the effect of taking $(e+1)\eta=1+\kappa$, instead of \eqref{eq:eta}.
Since
$$
\frac{(e+1)(\eta+1)}{2}=1+\frac{e+\kappa}{2}\in \ZZ,
$$
it follows from Lemma \ref{lem:3.5} that 
\begin{equation}\label{eq:curly-N1}
|S(\alpha)|\leq \frac{|P|^n\mathcal{N}^{2^{-d+1}}}{q^{(1+\kappa)(d-1)L}} ,
\end{equation}
where 
\begin{equation}\label{eq:curly-N2}
\mathcal{N}=
\#\left\{ \underline{\u}\in \cO^{(d-1)n}: 
\begin{array}{l}
|\u_1|,\dots,|\u_{d-1}|\leq q^\kappa\\
\|\alpha \Psi_i(\underline{\u}) \|< 
q^{\kappa(d-1)-de-1}
 ~(\forall   i\leq n)
\end{array}{}
\right\},
\end{equation}
Supposing that $\alpha=a/r+\theta$ for $a/r\in K$ and $\theta\in \TT$, our argument now bifurcates according to the degree of $r$.

\begin{lemma}[$\deg(r)\geq 1$]
\label{lem:pointwise'}
Let $L=2^{-d+1}n$, let  $a/r\in K$  and let $\theta\in \TT$. Assume that 
\begin{itemize}
\item[(i)]
$e\geq 1$, 
$q\leq |r|< q^{de+1-\kappa(d-1)}$ and 
$|r\theta|<q^{-\kappa (d-1)}$; or
\item[(ii)]
$e=1$, 
$q^2\leq |r|\leq q^{d}$ and
$|r\theta|\leq q^{-d}$.
\end{itemize}
Then  there exists a constant $c_{d,n}'>0$, independent of $q$, such that 
$$
|S(a/r+\theta)|\leq c_{d,n}' |P|^{n}q^{-L}.
$$
\end{lemma}

\begin{proof}
To deal with case (i) we apply 
 \cite[Lemma 4.3.5]{lee}
 (cf.~\cite[Lemma 3.6]{lee'}) 
with $Y=de+1-\kappa(d-1)$
and 
$M=\kappa(d-1)+\frac{1}{2}$. Our hypotheses ensure that $|r|<\hat Y$ and $|r\theta|<\hat M^{-1}$. Thus it follows that 
$\Psi_i(\underline{\u})\equiv 0\bmod{r}$ in \eqref{eq:curly-N2}, for all $i\leq n$. 
In particular we have $\mathcal{N}=0$ unless $\kappa=1$, which we now assume.

Pick a prime $\varpi\mid r$ with $|\varpi|\geq q$. If $|\varpi|\leq q^2$ we may break into residue classes modulo $\varpi$, finding that 
$$
\mathcal{N}\leq \sum_{\v_1,\dots,\v_{d-1}}
\#\left\{|\u_1|,\dots,|\u_{d-1}|\leq q: \u_i\equiv {\v_i} \bmod \varpi, \text{ for $1\leq i\leq d-1$}
\right\},
$$ 
where the sum is over all
$\underline{\v}=(\v_1,\dots,\v_{d-1})\in \FF_\varpi^{(d-1)n}$ such that 
$\Psi_i(\underline{\v})=0$, for all $i\leq n$,  over $\FF_\varpi.$ The inner cardinality is  $O((q^2/|\varpi|)^{(d-1)n})$, with an implied constant that is independent of $q$. We may use the Lang--Weil estimate to deduce that the outer sum is $O(|\varpi|^{(d-2)n})$, again with an implied constant that depends at most on $d$ and $n$. Hence we get the overall contribution
$$
\mathcal{N}\ll \frac{q^{2(d-1)n}}{|\varpi|^n}\leq  q^{2(d-1)n-n}.
$$
Alternatively, if $|\varpi|>q^2$, we may assume that the system of equations 
$\Psi_i=0$, for $i\leq n$, has dimension $(d-2)n$ over $\FF_\varpi.$
We now appeal to an argument of Browning and Heath-Brown \cite[Lemma 4]{41}.
Using induction on the dimension, as in the proof of \cite[Eq.~(3.7)]{41}, we easily conclude that 
$$\mathcal{N}\ll (q^{2})^{(d-2)n}\leq q^{2(d-1)n-2n},
$$ 
for an implied constant that only depends on $d$ and $n$.
Recalling that $\kappa=1$, 
 the first part of the lemma now follows on substituting these bounds into \eqref{eq:curly-N1}.

We now consider  case (ii), in which $e=1$, 
$q^2\leq |r|\leq q^{d}$ and $
|r\theta|\leq q^{-d}$.   Let $|a/r|=q^{-\alpha}$ for  $1\leq \alpha\leq d$. Let 
$v\in \{1,\dots, d\}$ be such that $d-v-\alpha=-1$. 
Then 
an application of Lemma \ref{lem:3.5 e=1} yields
\begin{align*}
|S(\alpha)|^{2^{d-1}}
&\leq |P|^{(2^{d-1}-d+1)n}q^{(v-1)n}M^{(v)}(\alpha)\\ 
&=
|P|^{2^{d-1}n}\cdot 
q^{(-2d+1+v)n}M^{(v)}(\alpha).
\end{align*}
Let $\underline{\u}\in\cO^{n(d-1)}$ be counted by $M^{(v)}(\alpha)$.
Since $|\theta|\leq q^{-d-2}$, it follows that
$|\theta\Psi_i(\underline{\u})|\leq  q^{-d-2}\cdot q^{d-v}=q^{-2-v}\leq q^{-3}$, for $1\leq i\leq n$. Similarly, 
for 
$1\leq i\leq n$, 
we have 
$|\frac{a}{r}\Psi_i(\underline{\u})|\leq q^{-\alpha}\cdot q^{d-v}=q^{-1}$.
If we write $\u_{j}=\u_j'+t\u_j''$, for $v\leq j\leq d$, where $\u_j', \u_j''\in \FF_q^n$, then the coefficient of $t^{-1}$ in the $t$-expansion of $\frac{a}{r}\Psi_i(\underline{\u})$ is equal to 
$\Psi_i(\u_1,\dots,\u_{v-1},\u''_v,\dots,\u_{d-1}'').
$ 
The condition 
$\|\alpha \Psi_i(\underline{\u}) \|<|P|^{-1}$ 
in $M^{(v)}(\alpha)$  implies that this coefficient must necessarily vanish, whence
$M^{(v)}(\alpha)$ is at most the number of 
$\u_1,\dots\u_{v-1},\u_{v}',\dots,\u_{d-1}', \u_{v}'',\dots,\u_{d-1}''\in\FF_q^n$ for which 
 $\Psi_i(\u_1,\dots,\u_{v-1},\u''_v,\dots,\u_{d-1}'')=0$, for $1\leq i\leq n$.
Thus $$M^{(v)}(\alpha)\ll q^{(d-v)n}\cdot q^{(d-2)n}=q^{(2d-v-2)n},$$
by the Lang--Weil estimate,
which implies the statement of  the lemma.
\end{proof}

\begin{lemma}[$\deg(r)=0$]
\label{lem:pointwise''}
Let $L=2^{-d+1}n$ and  let   $\theta\in \TT$. Assume that
$$
q^{-de-1}\leq |\theta|\leq q^{-1-\kappa (d-1)}.
$$
Then  there exists a constant $c_{d,n}''>0$, independent of $q$, such that 
$$
|S(\theta)|\leq c_{d,n}'' |P|^{n}q^{-L}.
$$
\end{lemma}

\begin{proof}
The upper bound assumed of $|\theta|$ implies that $|\theta\Psi_i(\underline{\u})|\leq q^{-1}$ in \eqref{eq:curly-N2}, for  $1\leq i\leq n$. Hence 
$\|\theta\Psi_i(\underline{\u})\|=|\theta\Psi_i(\underline{\u})|$ for $1\leq i\leq n$. 
Since $\alpha=\theta$ and $|\theta|\geq q^{-de-1}$, 
it follows that the condition 
$\|\alpha \Psi_i(\underline{\u}) \|< 
q^{\kappa(d-1)-de-1}$ is equivalent to 
$| \Psi_i(\underline{\u}) |<q^{\kappa(d-1)}.$
If $\kappa=0$ then it follows from \eqref{eq:curly-N2} that
$$
\mathcal{N}=
\#\left\{ \underline{\u}\in \FF_q^{(d-1)n}: 
 \Psi_i(\underline{\u})=0
 ~(\forall   i\leq n)
\right\}\ll q^{(d-2)n},
$$
by the Lang--Weil estimate. If, on the other hand, $\kappa=1$ then we write 
$\underline{\u}=\underline{\u}'+t \underline{\u}''$ in $\mathcal{N}$, under which transformation 
$| \Psi_i(\underline{\u}) |<q^{d-1}$ is equivalent to 
$\Psi_i(\underline{\u}'')=0$, for $i\leq n$.
Applying the Lang--Weil estimate to this  system of equations, we therefore  deduce that $\mathcal{N}=O(q^{(1+\kappa)(d-1)n-n})$ for $\kappa\in \{0,1\}$. An application of \eqref{eq:curly-N1} now completes the proof of the lemma.
\end{proof}

\section{The contribution from the minor arcs}

We assume that $d\geq 3$ throughout this section. Our goal
is to prove \eqref{eq:cup} for all $e\geq 1$, provided that $n> n_0(d)$,
where $n_0(d)$ is given by \eqref{eq:n0}.
The overall contribution to \eqref{eq:goat} from $|\theta|<q^{-3de}$ is easily seen to be negligible. 
Hence we may redefine the minor arcs to incorporate the condition 
$|\theta|\geq q^{-3de}$. For $\alpha,\beta\in \ZZ_{\geq 0}$, 
let $E(\alpha,\beta)$ denote the overall contribution to
 $N_{\text{minor}}(P)$,
from values of $a,r,\theta$ for which 
$|r|=q^\alpha$ and $|\theta|=q^{-\beta}.$ The contribution is empty unless
\begin{equation}\label{eq:basic}
0\leq \alpha\leq \frac{d(e+1)}{2} \quad \text{ and } \quad  \alpha+\frac{d(e+1)}{2}\leq \beta\leq 3de,
\end{equation}
with $\beta \leq de+1$ if $\alpha=0$. Since there are only finitely many choices of $\alpha,\beta$, in order to prove \eqref{eq:cup}, it will suffice to show that 
$$
\lim_{q\to \infty}q^{-\hat\mu} E(\alpha,\beta)=0,
$$
for each pair  $(\alpha,\beta)$ under consideration, assuming that 
 $n> n_0(d)$.
To begin with, summing trivially over $a$,
we have 
\begin{equation}\label{eq:to-return-to}
E(\alpha,\beta)\leq q^{2\alpha-\beta+1} \max_{\substack{a,r,\theta\\
|a|<|r|=q^\alpha\\ |\theta|=q^{-\beta}}}|S(a/r+\theta)|.
\end{equation}

We start by dealing with generic values of $\alpha$ and $\beta.$
 Lemma \ref{lem:pointwise} implies
that 
$$
E(\alpha,\beta)\leq c_{d,n}q^{2\alpha-\beta+1+(e+1)n-L(e+1)\eta},
$$
where $L=2^{-d+1}n$.
Recalling the definition \eqref{eq:hat-mu} of $\hat\mu$, the exponent of $q$ is $\hat\mu-\hat\nu$, with 
\begin{equation}\label{eq:nu-hat}
\begin{split}
\hat\nu
&=\{(n-d)e+n-1\}-\{
2\alpha-\beta+1+(e+1)n-L(e+1)\eta\}\\
&=L(e+1)\eta+\beta-de-2\alpha-2.
\end{split}
\end{equation}
For the choice of $\eta$ in \eqref{eq:eta}, 
and $n>n_0(d)$, 
we want to determine when  $\hat\nu>0$.
Returning to \eqref{eq:Gamma}, we now see that 
$$
\Gamma=\frac{1}{d-1}\min\left\{ (e+1)(d-1)-1, \beta-\alpha-1, (e+1)d-\alpha-1, \alpha+M\right\},
$$
where
$M=\max\{0,(e+1)d-\beta\}.$ 
The remainder of the  argument is  a case by case analysis. When $[\Gamma]\leq 1$ we shall return to \eqref{eq:to-return-to}, and  argue differently based instead on Lemmas \ref{lem:pointwise'}
and 
\ref{lem:pointwise''}.

\subsection*{Case 1: $\alpha\geq 2(d-1)$ and $\beta\geq (e+1)d+1$}

In this case $M=0$. Using \eqref{eq:basic}, one finds that 
$$
\Gamma=\frac{1}{d-1}\times \begin{cases}
\alpha &\text{ if $\alpha<\frac{d(e+1)}{2}$},\\
\alpha-1 &\text{ if $\alpha=\frac{d(e+1)}{2}$.}
\end{cases}
$$
Let $\iota\in \{0,1\}$. We write
$
\alpha-\iota=k(d-1)+\ell,
$
for  $k\in \ZZ_{\geq 0}$ and $\ell\in \{0,\dots,d-2\}$. 
Then \eqref{eq:eta} implies that 
$(e+1)\eta=k-\delta$, where
\begin{equation}\label{eq:hurry}
\delta= \begin{cases}
0 & \text{ if $k\not\equiv e\bmod{2}$,}\\
1 & \text{ if $k\equiv e\bmod{2}$.}
\end{cases}
\end{equation}
We claim that the assumption $\alpha\geq 2(d-1)$ implies that $k\geq 2$, or else $k=1$ and $\delta=0$. 
This is obvious when 
$\alpha<\frac{d(e+1)}{2}$. Suppose that  $k=1$ and $\alpha=\frac{d(e+1)}{2}$. Then $\iota=0$ and
$\ell=d-2$, whence  $\alpha=2(d-1)=\frac{d(e+1)}{2}$. Since $d\geq 3$, this equation has no solutions in odd integers $e$. Thus $\delta=0$.

Recalling \eqref{eq:nu-hat} and substituting for $\alpha$, we find that 
\begin{align*}
\hat \nu
&=L(k-\delta)+\beta-de-2 
-2\iota-2k(d-1)-2\ell\\
&=(L-2(d-1))k-\delta L
+\beta-de-2 
-2\iota-2\ell\\
&\geq (L-2(d-1))k-\delta L
-d+3 
-2\iota,
\end{align*}
since  $\beta\geq (e+1)d+1$ and $\ell\leq d-2$.
Taking $3-2\iota\geq 0$, we have therefore shown that 
$\hat\nu\geq \hat\nu_0$, with 
$$
 \hat\nu_0=
 (L-2(d-1))k-\delta L
-d.
$$
If $k\geq 2$, then we take $\delta\leq 1$ to conclude that
$$
\hat\nu_0\geq (2-\delta)L-4(d-1)-d\geq 
L-5d+4.
$$
Thus $\hat\nu_0>0$ if $n>n_0(d)$.
Alternatively, if $k=1$ then we must have $\delta=0$. 
It follows that
$$
\hat\nu_0= L-2(d-1)-d=
L-3d+2,
$$
whence $\hat\nu_0>0$ if 
$n>n_0(d)$, since $n_0(d)\geq 2^{d-1}\cdot (3d-2)$ in \eqref{eq:n0}.

\subsection*{Case 2: 
$\alpha +de-d+2> \beta$
 and $\beta\leq (e+1)d$}

In this case $M=(e+1)d-\beta$. It follows from  \eqref{eq:basic}
that 
$$
\Gamma=\frac{1}{d-1}\times \begin{cases}
\alpha+(e+1)d-\beta &\text{ if $\beta>2\alpha$},\\
\alpha+(e+1)d-\beta-1 
&\text{ if $\beta\leq 2\alpha$}.
\end{cases}
$$
We proceed as before. Thus for  $\iota\in \{0,1\}$, we write
\begin{equation}\label{eq:sheep}
\alpha+(e+1)d-\beta-\iota=k(d-1)+\ell,
\end{equation}
with  $k\in \ZZ_{\geq 0}$ and $\ell\in \{0,\dots,d-2\}$. 
Then \eqref{eq:eta} implies that 
$(e+1)\eta=k-\delta$, where $\delta$ is given by \eqref{eq:hurry}.
If $k\geq 2$ then \eqref{eq:nu-hat} yields
\begin{align*}
\hat \nu
&=L(k-\delta)-\beta+de-2 
-2\iota-2k(d-1)-2\ell+2d\\
&=(L-2(d-1))k-\delta L
-\beta+de-2 
-2\iota-2\ell+2d\\
&\geq L-4d+4-\beta+de,
\end{align*}
since  $\delta,\iota\leq 1$ and $\ell\leq d-2$.
But $\beta\leq (e+1)d$, and so it follows that 
$\hat\nu \geq L-5d+4$, which is positive if $n>n_0(d)$.
Suppose that $k\leq 1$. Then, on taking $\iota\leq 1$ and $\ell\leq d-2$ in \eqref{eq:sheep}, we must have that 
$$
\alpha +de-d+2\leq \beta,
$$
which contradicts the hypothesis.

\subsection*{Case 3: $\alpha\leq 2(d-1)$ and $\beta\geq (e+1)d+1$}

In this case we return to \eqref{eq:to-return-to}, and we recall the definition \eqref{eq:kappa} of $\kappa$. 
Suppose first that $\alpha=0$.
It follows from Lemma \ref{lem:pointwise''} that 
$S(a/r+\theta)\ll |P|^n q^{-L}$ if
$$
1+\kappa(d-1)\leq \beta \leq de+1.
$$
The upper bound $\beta\leq de+1$  follows from the definition of the minor arcs when $\alpha=0$. Moreover, the lower bound holds, since for $e\geq 1$ it follows from \eqref{eq:basic} that
$
\beta\geq  d\geq 1+\kappa(d-1).$
Recalling \eqref{eq:hat-mu}, we conclude that 
$$
E(\alpha,\beta)\ll q^{-\beta+1+(e+1)n-L}=q^{\hat\mu-\hat\nu},
$$
with $\hat\nu=L+\beta-de-2\geq L>0$, which is satisfactory. 

Suppose next that $\alpha\geq 1$. Then 
$S(a/r+\theta)\ll |P|^n q^{-L}$, by 
Lemma \ref{lem:pointwise'}, provided  that 
\begin{equation}\label{eq:hungry}
e\geq 1, \quad
1\leq \alpha <de+1-\kappa(d-1) \quad \text{ and } \quad \alpha-\beta<-\kappa(d-1),
\end{equation}
or 
\begin{equation}
 \label{eq:hungrier}
 e=1, \quad  2\leq \alpha\leq d \quad \text{ and } \quad \alpha-\beta\leq -d.
\end{equation}

In view of \eqref{eq:basic}, it is easily seen that $\alpha-\beta<-(d-1)\leq -\kappa(d-1)$.
Next, we claim that $2d-2<de+1-\kappa(d-1)$ for any $e\geq 2$. 
This is enough to confirm \eqref{eq:hungry}, since $\alpha\leq 2(d-1)$.
The claim is obvious when   $\kappa=1$ and $e\geq 3$.
On the other hand, if $\kappa=0$ then $e\geq 2$ and it is clear that 
$2d-2\leq 2d+1\leq de+1$. 
Next, suppose that  $e=1$, so that $\kappa=1$.
If $\alpha=1$ then we are plainly in the situation covered by
\eqref{eq:hungry}. If $\alpha\geq 2$, on the other hand,  then 
\eqref{eq:basic} implies that $\alpha\leq d$ and $\alpha-\beta\leq -d$, so that we are in the case covered by \eqref{eq:hungrier}.
It follows that 
$$
E(\alpha,\beta)\ll q^{2\alpha-\beta+1+(e+1)n-L}=q^{\hat\mu-\hat\nu},
$$
with 
\begin{align*}
\hat\nu=L+\beta-de-2-2\alpha
&\geq  L+d-1-2\alpha\\
&\geq L-3d+3,
\end{align*}
since $\alpha\leq 2(d-1)$ and $\beta\geq (e+1)d+1$.
This is positive for $n>n_0(d)$.  

\subsection*{Case 4: $\alpha+de-d+2\leq \beta$ and $\beta\leq (e+1)d$}

We begin as in the previous case. If  $\alpha=0$, the same argument goes through, leading to 
$
E(\alpha,\beta)\ll q^{\hat\mu-\hat\nu},
$
with $\hat\nu=L+\beta-de-2\geq L-d.$ This is certainly  positive for $n>n_0(d)$. 
Suppose next that $\alpha\geq 1$. Then 
$S(a/r+\theta)\ll |P|^n q^{-L}$, by 
Lemma \ref{lem:pointwise'}, provided  that \eqref{eq:hungry} or \eqref{eq:hungrier} hold.  Note that
$$
\alpha\leq \beta-de+d-2\leq 2d-2< 
de+1-\kappa(d-1),
$$
for any $e\geq 2$, 
by the calculation in the previous case. Likewise, the previous argument shows that we are covered by \eqref{eq:hungry} or \eqref{eq:hungrier} when $e=1$.
Thus we find that 
$E(\alpha,\beta)\ll q^{\hat\mu-\hat\nu},$
with 
\begin{align*}
\hat\nu=L+\beta-de-2-2\alpha
&\geq  L-d-\alpha\\
&\geq L-3d+2,
\end{align*}
since $\alpha\leq 2(d-1)$.
This is also  positive for $n>n_0(d)$.

\end{document}